\documentclass[11pt]{amsart}
\usepackage[paper=a4paper, text={155mm,218mm},centering]{geometry}
\usepackage{tikz}
\usetikzlibrary{intersections,patterns,matrix, decorations.markings,decorations.pathreplacing,arrows}
\usepackage{pgfplots}
\usepackage{mathabx}
\usepackage{url}
\usepackage{amsmath}
\usepackage{amssymb}
\usepackage{amscd}
\usepackage{manfnt}
\usepackage{enumitem}
\usepackage[all]{xy}
\usepackage{graphicx}
\usepackage{MnSymbol}
\usepackage{comment}

\numberwithin{equation}{section}

\newtheorem{theorem}{Theorem}[section]
\newtheorem*{theorem*}{Theorem}
\newtheorem{proposition}[theorem]{Proposition}
\newtheorem{lemma}[theorem]{Lemma}
\newtheorem{corollary}[theorem]{Corollary}

\theoremstyle{definition}
\newtheorem{definition}[theorem]{Definition}

\theoremstyle{remark}
\newtheorem{remark}[theorem]{Remark}
\newtheorem*{Ack}{Acknowledgments}
\def\wt#1{\widetilde{#1}}

\def\ol#1{\overline{#1}}

\def\pind{\ind_{\partial}}

\usepackage{etoolbox}
\def\do#1{\csdef{c#1}{\mathcal{#1}}}
\docsvlist{A,B,C,D,E,F,G,H,I,J,K,L,M,N,O,P,Q,R,S,T,U,V,W,X,Y,Z}
\def\do#1{\csdef{#1}{\mathbb{#1}}}
\docsvlist{D,N,Z,Q,R,F,A,W,M}
\def\do#1{\csdef{#1#1}{\mathbf{#1}}}
\docsvlist{A,B,C,D,E,F,G,H,I,J,K,L,M,N,O,P,Q,R,S,T,U,V,W,X,Y,Z,a,b,c,d,e,f,h,i,j,k,m,n,o,p,q,r,s,t,u,v,w,x,y,z}
\def\do#1{\csdef{#1}{\operatorname{#1}}}
\def\do#1{\csdef{m#1}{\mathfrak{#1}}}
\docsvlist{a,b,j,m,E,I,J}
\def\do#1{\csdef{#1}{\operatorname{#1}}}
\docsvlist{atan,codim,Crit,dist,Fix,Hom,Id, ind, Int,lk, Mod,Sym, supp,sgn}

\newcounter{nparcount}

\makeatletter

\def\@setref#1#2#3{%
  \ifx#1\relax
   \protect\G@refundefinedtrue
   \nfss@text{\colorbox{green!30}{[REF]}}%
   \@latex@warning{Reference '#3' on page \thepage \space
             undefined}%
  \else
   \expandafter#2#1\null
  \fi}
\def\@citex[#1]#2{\leavevmode
  \let\@citea\@empty
  \@cite{\@for\@citeb:=#2\do
    {\@citea\def\@citea{,\penalty\@m\ }%
     \edef\@citeb{\expandafter\@firstofone\@citeb\@empty}%
     \if@filesw\immediate\write\@auxout{\string\citation{\@citeb}}\fi
     \@ifundefined{b@\@citeb}{\colorbox{blue!30}{\reset@font REF}%
       \G@refundefinedtrue
       \@latex@warning
         {Citation `\@citeb' on page \thepage \space undefined}}%
       {\@cite@ofmt{\csname b@\@citeb\endcsname}}}}{#1}}
\makeatother

\title{Families of Morse functions for manifolds with boundary}

\author{Maciej Borodzik}
\address{Institute of Mathematics, University of Warsaw, ul. Banacha 2,
02-097 Warsaw, Poland}
\email{mcboro@mimuw.edu.pl}

\author{Weronika Buczyńska}
\address{Institute of Mathematics, University of Warsaw, ul. Banacha 2,
02-097 Warsaw, Poland}
\email{wkrych@mimuw.edu.pl}

\makeatletter
\@namedef{subjclassname@2010}{\textup{2010} Mathematics Subject Classification}
\makeatother

\begin{document}

\begin{abstract}
  We study 1-parameter families of Morse functions for manifolds with boundary. We list all degeneracies that may occur
  in generic 1-parameter families.
\end{abstract}

\maketitle
\section{Introduction}
Morse theory for manifolds with boundary was developed by Hajduk \cite{Haj} and Jankowski--Rubinsztein \cite{JR}
in the 1970s, and rediscovered by 
Kronheimer and Mrowka 
\cite{KM}. Despite formal similarities to the classical case of closed manifolds, the theory is more complex.
The analog of Morse--Smale--Witten chain complex for manifolds with boundary involves non-trivial terms coming from
critical points on the boundary; see \cite{Haj,KM}.
In \cite{BNR} it was proved  that a critical point in the interior can be pushed into the boundary, where it splits into two boundary critical points. This phenomenon, which we will refer to as \emph{boundary splitting} was observed first by \cite{Haj}. In \cite{BMi}, an opposite phenomenon was studied: given two boundary critical points such that there is a single trajectory of a gradient-like vector field connecting them, one can collide these critical points and push them into the interior. We refer to \cite{BMi} for a precise statement.

Families of Morse functions on closed manifolds were extensively studied by Cerf \cite{Cerf}. However, even one-parameter families of Morse functions on manifolds with boundary have not been described in the literature. A natural question arises: Is it true that a typical one-parameter family of functions consists of Morse functions and finitely many occurrences of births/deaths and boundary splittings?

In this paper, we answer this question negatively. We study 1-parameter families of smooth functions on manifolds with boundary and investigate non-Morse singularities that appear generically (there is a subtlety in what we mean by a generic $1$-parameter family, see Remark~\ref{rem:subtle}). Among them, as expected, there are birth/deaths of pairs of critical points in the interior, births/deaths of pairs of boundary critical points. However, there is an extra phenomenon: an interior critical point collides with a boundary critical point, and then the critical point on the boundary changes its stability. Surprisingly, this phenomenon -- unlike boundary splitting --  is a codimension~1 behavior. We call it a \emph{collision} and describe in detail in Subsection~\ref{sub:beefurcate}.

Our main theorem describes the full list of  bifurcations that can occur in generic one-parameter families. It is stated as follows:
\begin{theorem}\label{thm:main}
  Let $N$ be a smooth compact manifold with boundary.
	Suppose $F_0,F_1\colon N\to\R$ be two Morse functions.
	Assume $F_\sigma$, $\sigma\in[0,1]$ is a path connecting these functions. Then, up to perturbing $F_\sigma$ rel $F_0,F_1$, we can assume that
	\begin{itemize}
		\item there exists a finitely many points $\sigma_1,\dots,\sigma_m$ such that $F_\sigma$ is Morse if $\sigma$ is different from $\sigma_1,\dots,\sigma_m$;
		\item at each of $\sigma_i$, there is precisely one non-Morse critical pint. The change between $F_{\sigma_i-\varepsilon}$ and $F_{\sigma_i+\varepsilon}$ (with $\varepsilon>0$ sufficiently small) is one of the following:
			\begin{itemize}
				\item birth/death of critical points in the interior;
				\item birth/death of critical points on the boundary;
				\item a collision of a boundary and interior critical point.
			\end{itemize}
	\end{itemize}
\end{theorem}
We refer to Subsection~\ref{sub:beefurcate} for a detailed description of these three events. As an example,
we will describe the boundary splitting
as a codimension~2 behavior and show how it can be presented as a sequence of codimension~1 events, see Subsection~\ref{sub:moving}.

The main idea of the proof is to pass from a function on $N$ to a function on its double and consider it as a $\Z_2$-equivariant function. Then, we study codimension~1 singularities of $\Z_2$-equivariant functions via equivariant Thom transversality theorem
as in \cite{Wall_jet}.

The structure of the paper is the following. In Section~\ref{sec:bmorse},
we recall basic properties of Morse functions for manifolds with boundary. We notice that there are essentially three non-equivalent notions of a Morse function for manifolds with boundary. We give a translation between the two of the three notions.
Finally, we recall equivariant jet theory of \cite{Wall_jet}. It gives a framework for stating and proving genericity functions. Section~\ref{sec:cerf} develops theory of families of equivariant Morse functions. Theorem~\ref{thm:main} is proved in Section~\ref{sec:pmain}. As an example, we show how boundary splitting can be decomposed into a birth  and a collision in Subsection~\ref{sub:moving}.

\begin{Ack}
	The authors would like to thank Mark Powell and Stefan Friedl for fruitful conversations. The first author was covered by the Polish NCN Opus grant 2024/53/B/ST1/03470. 
\end{Ack}

\section{Boundary Morse functions}\label{sec:bmorse}
\subsection{Definitions}\label{sub:defini}
Given a closed manifold $X$, a function $F\colon X\to\R$ is called \emph{Morse}, if for
any $p\in\Crit(F)$, $\det D^2F(p)\neq 0$.
Hereafter, $\Crit(F)$ is the set of critical points of a function. If $DF(p)=0$, then the condition $\det D^2F(p)\neq 0$
is independent on the choice of coordinates.

Following \cite{BNR} we extend this definition for manifolds with boundary.
Let $N$ be a smooth manifold with boundary.
\begin{definition}\label{def:bmorse}
  A function $F\colon N\to\R$ is \emph{boundary Morse} if
  \begin{enumerate}[label=(BM-\arabic*)]
    \item for any $p\in N$, if $DF(p)=0$, then $\det D^2F(p)\neq 0$;\label{item:nondeg}
    \item for any $p\in\partial N$, with $f:=F|_{\partial N}$, if $Df(p)=0$, then $DF(p)=0$.\label{item:partial}
  \end{enumerate}
\end{definition}
For future reference, we recall the boundary Morse lemma, see \cite[Lemma 2.6]{BNR}.
\begin{lemma}\label{lem:bmorse}
  Suppose $p\in\partial N$ is a critical point of a boundary Morse function. Then, there are local coordinates near $p$,
  $x,y_1,\dots,y_{n-1}$, and choices of signs $\epsilon_x,\epsilon_1,\dots,\epsilon_{n-1}$, such that $\partial N=\{x=0\}$ and
  \begin{equation}\label{eq:bmorse}F(x,y_1,\dots,y_{n-1})=F(p)+\epsilon_x x^2+\sum_{i=1}^{n-1}\epsilon_i y_i^2.\end{equation}
\end{lemma}

We pause for the moment to clarify different variants of Definition~\ref{def:bmorse}.
Hajduk's $m$-functions satisfy~\ref{item:nondeg}, but a critical point of $f$ is assumed \emph{not to be} a critical point of $F$,
so the opposite of \ref{item:partial} is assumed. That is, $m$-functions share more properties of embedded Morse functions, see \cite{BP}.
In particular, a boundary Morse lemma for $m$-function should have term $x^2$ replaced by $x$, as in the case of embedded Morse
functions, compare \cite[Lemma 2.16]{BP}.

We will see below, 
Hajduk's $m$-function are more natural in the sense that the set of $m$-functions in open-dense, unlike the set of boundary Morse functions.
As a drawback, 
analogs of gradient-like vector fields in that case have to use embedded gradient-like vector fields, see \cite[Section 3]{BP}.
The latter
are not linearizable at critical points on $\partial N$. This leads to a more complicated local behavior of the trajectories of gradient-like vector fields for $m$-functions.

In \cite[Section 2.4]{KM}, Kronheimer and Mrowka use another definition, namely they consider the double $M$ of $N$ and look at functions
on $N$ that are restrictions of $\Z_2$ equivariant functions on $M$.
This approach is suitable for studying degenerations of Morse functions, via equivariant jet transversality theorem. As the main drawback we note that not all `regular' functions on $N$ extend to smooth symmetric functions on $M$. So the definitions seems rather restrictive.

Definition~\ref{def:bmorse} combines drawbacks and highlights of the two other approaches. Namely, studying gradient-like vector fields
is the easiest, and the conditions of Definition~\ref{def:bmorse} are less restrictive than those in \cite{KM}. On the other
hand, as we shall see in Proposition~\ref{prop:bummer}, functions satisfying \ref{item:partial} are not open-dense in the space of all Morse functions. This means that
studying families of Morse functions via transversality theory is substantially harder.

In Subsection~\ref{sub:doublable}, we explain how to pass from a boundary Morse function of Definition~\ref{def:bmorse} to
a Morse function in the sense of \cite{KM}. The passage will be done in a precise way, in particular not modifying the function near 
boundary critical points.  We will then study deformations of equivariant functions leading to the proof of Theorem~\ref{thm:main}.

We conclude Subsection~\ref{sub:defini} with the following observation. For a manifold with boundary $N$, denote $\cF(N)=C^\infty(N;\R)$, and let $\cB(N)\subset\cF(N)$
be the subspace of boundary Morse functions.
\begin{proposition}\label{prop:bummer}
  If $\partial N\neq \emptyset$, then $\cB(N)$ is neither open nor dense in $\cF(N)$.
\end{proposition}
\begin{proof}
  The culprit is the specific condition~\ref{item:partial}. Indeed, if $p\in\partial N$ is such that $DF(p)=0$, then
  \ref{item:partial} imposes that $\partial_{\vec{n}}F(p)=0$, where $\vec{n}\in T_pN\setminus\{0\}$ is normal to $\partial N$. This is a closed
  condition, and it
  leads to failure of both openness and density. 

  To see that failure, we construct a function that satisfies \ref{item:nondeg}, but not \ref{item:partial}, and show
  that no function in its neighborhood satisfies \ref{item:partial}.
  Suppose $F\colon N\to\R$ is a boundary Morse function and $p\in\partial N$ is a critical point. Choose
  local coordinates $(x,y_1,\dots,y_n)$ near $p$, so that $\partial N=\{x=0\}$ and $F$ has form 
  \eqref{eq:bmorse}. Choose a smooth function $\eta\colon N\to[0,1]$
  near $p$, equal to $1$ in a neighborhood of $p$ and supported in a set where the local coordinates are defined.
  Let $\wt{F}=F+\varepsilon x\eta$. Then, $\wt{F}$ fails to \ref{item:partial}, indeed,
  there is a neighborhood $U$ of $p$ such that $\frac{\partial\wt{F}}{\partial x}\ge\frac23\varepsilon$ on $U$.

  Suppose $F'$ is close to $\wt{F}$, that is $\dist_{C^2}(F',\wt{F})<\delta$ for $\delta>0$ sufficiently small. In particular, if $\delta<\frac{\varepsilon}{3}$, then
  $\frac{\partial F'}{\partial x}>\frac13\varepsilon$ everywhere on $U$. Next, $\dist_{C^2}(F'|_{\partial N},F|_{\partial N})<\delta$. The function $F|_{\partial N}$
  is Morse in the ordinary sense, which is an open condition in the $C^2$-norm. If $\delta$ is sufficiently small,
  then $F'|_{\partial N}$ must have a Morse critical point in $U\cap\partial N$.
  Hence, $F'$ is not Morse, because it does not satisfy \ref{item:partial}.

  This shows that an open neighborhood of $F'$ consists of non-Morse functions, leading to the failure of density. Also, by taking
  $\wt{F}_n=F+\frac1n x\eta$, we obtain a sequence of functions failing to \ref{item:partial} converging to the Morse function $F$.
  That is, $\cB(N)$ is not open.
\end{proof}
\begin{remark}\label{rem:subtle}
  Failure to openness and density of boundary Morse functions makes it hard to define precisely what is a \emph{generic path of boundary Morse
  functions}, Formally one would have to consider e.g. generic functions withing the closure of boundary Morse functions in the space of all functions. We do not find applications for  this rather delicate notion. In the present paper, we consider mostly paths of functions
  on $N$ that are restrictions of paths of equivariant functions on the double of $N$.
\end{remark}
\subsection{Index of a critical point}\label{sub:index}
Suppose $F\colon N\to\R$ is a Morse function for a manifold with boundary.
Assume $p$ is a critical point of $F$.

The index of a critical point $p$, $\ind p$ is by definition the dimension
of the negative subspace of $D^2F(p)$: the index of local mininum is equal to $0$,
the index of a local maximum is equal to $\dim N$.

If $p\in\partial N$ is a critical point, we can also consider the restriction $f:=F|_{\partial N}$ and consider
$\pind p$ to be the dimension of the negative subspace of $D^2f(p)$. Informally,
$\ind p=\pind p$ if the function $N$ is increasing when going from $p$ to
the interior, while $\ind p=\pind p+1$ if $F$ is decreasing in the normal direction.

\begin{definition}\label{def:inddef}
	A critical point $p\in\partial N$ is called \emph{boundary stable} if
	$\ind p=\pind p+1$. It is called \emph{boundary unstable} if $\ind p=\pind p$; see Figure~\ref{fig:bstable}.
\end{definition}
\begin{figure}
  \begin{tikzpicture}
    \draw (2,0) node {\includegraphics[width=2.5cm]{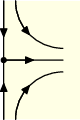}};
    \draw (-2,0) node {\includegraphics[width=2.5cm]{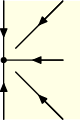}};
  \end{tikzpicture}
  \caption{Boundary stable (left) and boundary unstable (right) critical points. We have sketched the flow of the gradient of
  the corresponding Morse function.}\label{fig:bstable}
\end{figure}
In \cite{KM,BNR}, boundary stability was defined using gradient vector fields. In that setting, a
critical point $p$ is boundary stable if the unstable manifold of $p$ is contained
in $\partial N$. Here, unstable manifold is with respect to vector field $\nabla F$
(\cite{KM} uses $-\nabla F$). The two definitions are equivalent. In fact,
if $\ind p=\pind p$, then the negative definite part of $D^2F(p)$ is
contained in $T_p\partial N$. This implies that the tangent space to the stable
manifold of $\nabla F$ at $p$ is contained in $T_p\partial N$. As \cite{KM,BNR} it was assumed that $\nabla F$ is tangent to $\partial N$, it quickly follows that $W^s(p)$ is contained in $\partial N$, so the critical point is boundary unstable. 

The advantage of Definition~\ref{def:inddef} is that it does not require introducing any auxilliary objects such as a Riemannian metric or a gradient-like vector field.

\subsection{Doublable functions}\label{sub:doublable}
We now realize our plan to compare boundary Morse functions of Definition~\ref{def:bmorse} with
the approach of \cite{KM}. 
Let $N$ be a manifold with boundary. The \emph{double} of $M:=\cD(N)$ is the boundary connected sum of a manifold with a copy itself, $\cD(N)=N\cup_{\partial N} N$. To form a double, we fix an identification $\tau$ of $N$ with its copy. This identification leads to an involution $\tau\colon M\to M$.
With that choice, $M$ has a structure of a $\Z_2$ manifold. The fixed point
set $M^{\Z_2}$ is $\partial N$. 

Any continuous function $P\colon N\to\R$ can be extended to a function $\Phi\colon N\to\R$ by declaring that $\Phi(z)=\Phi(\tau z)=P(z)$,
for all $z\in N$. The extension $\Phi$ is continuous, but not necessarily smooth.
\begin{definition}
  A function $P\colon N\to\R$ is called \emph{doublable} if the extension $\Phi$ of $P$ to $M$ is a smooth function.
\end{definition}
We define $\cP(N)$ to be the set of doublable functions. 
\begin{proposition}\label{prop:dbldbl}
  Suppose $F\colon N\to\R$ is a boundary Morse function. Then, there exists a function $P\colon N\to\R$ with the following properties.
  \begin{enumerate}[label=(F-\arabic*)]
    \item $P$ is doublable;\label{item:Fdouble}
    \item $\dist_{C^0}(F,P)$ can be as small as we please, moreover $F$ agrees with $P$ away from an arbitrarily small
      neighborhood of $\partial N$;\label{item:Fsmall}
    \item $P$ has the same set of critical points as $F$;\label{item:Fcrit}
    \item $P$ agrees with $F$ in a neighborhood of these critical points;\label{item:Fagr}
    \item there exists a family $F_\sigma$, $\sigma\in[0,1]$ of boundary Morse functions, such that
      $F_0=F$, $F_1=P$ and $F_\sigma$
      is independent of $\sigma$ near any critical point of $F$.\label{item:Ffam}
  \end{enumerate}
\end{proposition}
\begin{proof}
  By the collar theorem, there exists a neighborhood $V$ of $\partial N$ such that $V$ can be identified
  with $\partial N\times[0,\varepsilon)$.
  We write $\theta\colon V\to \partial N\times[0,\varepsilon)$.
  
  For each critical point $p\in\partial N$ of $F$ we fix  a neighborhood  $U_p$ on which $F$
  has local coordinates as in  boundary Morse Lemma~\ref{lem:bmorse}.
  We assume  the identification $\theta$ takes $U_p\cap V$ to
  $\left( U_p\cap\partial N \right) \times[0,\varepsilon]$ via $\theta(x,y_1,\dots,y_{n-1})=(0,y_1,\dots,y_{n-1})\times\{x\}$.

  This specific property can be achieved by constructing  $\theta$  as follows:  choose a Riemannian metric
  and  for each $z\in N$ consider the nearest point $z'\in\partial N$ on the boundary (with the smallest distance).
  By the implicit function theorem  this assignment is well-defined on a neighborhood $V$ of $\partial N$ (the
  neighborhood depends on the metric). We define $\theta(z)=(z',\dist(z,z'))$.

  Having fixed the identification, we use Hadamard Lemma to write
  \[F(z,x)=F(z,0)+xG(z,x).\]
  for $(z,x)\in \partial N\times[0,\varepsilon)$, where $G$ is a smooth function.
  Now choose a bump function  $\eta\colon N\to[0,1]$ equal to $0$ away from $\partial N\times[0,\varepsilon/2]$. 
  Our aim is to set
  \begin{equation}\label{eq:pdef}
    P(z,x)=F(z,0)+(1-\eta)xG(z,x).
  \end{equation}
  That is, $P(z,x)=F(z,x)$ on $\partial N\times[\varepsilon/2,\varepsilon)$ and setting $P$ equal to $F$
  away from $V$ leads to a smooth function.
  The choice of $\eta$ has to be made carefully, since we want $P=F$ near critical points of $F$.

  For a critical point $p\in\partial N$,  let $U'_p$ be a neighborhood such that $\ol{U'_p}\subset U_p$
  and $U'_p$ has product structure,
  $U'_p=(U'_p\cap\partial N)\times[0,\delta)$, where $\delta>\varepsilon$
  (we can first choose $\delta$ and then $\varepsilon$).
  Define $\eta$ to be $\eta_0\times\rho$, where $\eta_0\colon \partial N\to\R$,\
  and $\rho\colon[0,\varepsilon)\to[0,1]$ are such that
  \begin{enumerate}[label=(N-\arabic*)]
    \item $\rho\equiv 1$ on $[0,\varepsilon/4]$;\label{item:rhoconst}
    \item $\rho\equiv 0$ away from $[0,\varepsilon/2]$;
    \item $\eta_0\equiv 1$ away from $\bigcup U_p$;
    \item $\eta_0\equiv 0$ on $\bigcup U'_p$.
  \end{enumerate}
  Furthermore, we can assume that $|D\eta|<C_\eta\varepsilon^{-1}$ for some constant
  $C_\eta>0$. That constant does not depend on $\varepsilon$, but can depend on $U_p$ and $U'_p$, in general
  it might depend on the distance between boundaries of $U_p$ and $U'_p$. Now we define $P$ via \eqref{eq:pdef},
  and check all the required properties.
  
  \emph{Checking \ref{item:Fdouble}.} 
  A smooth function $G\colon N\to\R$ is doublable if its restriction to $\partial N\times[0,\varepsilon)$ satisfies
  $\frac{\partial^k G}{\partial x^k}(z,0)=0$ for all $z\in\partial N$ and for all odd $k\ge 1$.
 Our  $P$ satisfies this condition for all $k$ odd.
 Indeed,  if $z\notin  \bigcup U_p$, then $P(z,x)=P(z,0)$ for all $x$ sufficiently small. 
  
  Suppose $z\in U_p$. The specific form of $\eta=\eta_0\rho$ together with \ref{item:rhoconst},
  implies that $\frac{\partial^k}{\partial x^k}\eta(z,0)=0$ for all $x$. Furthermore, by boundary Morse lemma,
  $xG(z,x)=\pm x^2$,
  where the sign depends on boundary stability of the critical point $p$.
  It is routine to check that $\frac{\partial^k P}{\partial x^k}(z,0)=0$
  for all $k\neq 0,2$.

  \emph{Checking \ref{item:Fsmall}.} Notice that $\dist_{C^0(N;\R)}(F,P)\le \sup t|G(z,t)|$, where
  the supremum is taken over the $(z,t)\in\partial N\times[0,\varepsilon)$.
  To show that this can be made as small as we please, notice
  that by the mean value theorem applied to $x\mapsto F(z,x)$
 $G(z,x) = \frac{\partial F}{\partial x} (z,\zeta_x)$ for some $\zeta_x\in(0,x)$. This means that $||G(z,t)||\le ||F(z,t)||_{C^1}$.
  Hence, the distance between $F$ and $P$ in the $C^0$-norm is bounded by $\varepsilon||F||_{C^1}$.

  \emph{Checking \ref{item:Fcrit} and \ref{item:Fagr}}. The number of critical points of $F$ is finite, hence,
  for $\varepsilon>0$ sufficiently small, the region $\partial N\times[0,\varepsilon)$ contains
  only critical points of $F$ on $\partial N$. From this we conclude that there is a global constant $c$
  such that $||DF(z,0)||>c$ for all $z\in \partial N\times[0,\varepsilon)\setminus\bigcup U'_p$. Suppose $||DG||<C$.
  On choosing $\varepsilon$ such that $\varepsilon C<c$, we achieve that the derivative of $P$
  in the $z$-directions is non-zero. That is, $P$ has no critical points on $\partial N\times[0,\varepsilon)$
  away from $\bigcup U'_p$, but on each $U'_p$, $P=F$.
  Given \ref{item:Fcrit}, item~\ref{item:Fagr} is trivial, because $P=F$ near critical points of $F$.

  \emph{Checking \ref{item:Ffam}}.
  We define $F_\sigma(z,x)=F(z,0)+(1-\sigma\eta(z,x))xG(z,x)$. The proof that $F_\sigma$ has no critical
  points other than $F$ follows the same pattern as the proof of \ref{item:Fcrit} above.
  Also, $F_\sigma=F$ on each of the $U'_p$, $F_0=F$
  and $F_1=P$.
\end{proof}

\subsection{Equivariant Morse theory}
Morse theory for functions on manifolds with group actions was studied by Wassermann \cite[Section 4]{Was}.
We quote a few results, restricted to our specific situation of $G=\Z_2$, $\dim M^{\Z_2}=\dim M-1$. General results can be found in \cite{Was}.
Throughout the subsection, we assume that $M$ is a smooth closed manifold, acted upon by $G=\Z_2$.
We denote by $\tau\colon M\to M$ the involution.
\begin{definition}
  Suppose $\Phi\colon M\to\R$ is an equivariant function.
  We say $f$ is \emph{equivariant Morse} if for any critical point $p$ of $\Phi$, $D^2\Phi(p)$ is non-degenerate.
\end{definition}
The definition is analogous to all previous definitions of Morse functions.
\begin{lemma}\label{lem:doubl}
  Suppose $P\colon N\to\R$ is a doublable function. Let $M$ be the double of $N$ and  $\Phi\colon M\to\R$ be 
  the extension of $P$ to a $\Z_2$-invariant function on $M$.
  Then $P$ is boundary Morse if and only if the function $\Phi$ is equivariant Morse.
\end{lemma}
\begin{proof}
  The proof is a routine checking of the axioms, we leave them to the reader.
\end{proof}

 There is an equivariant version of Morse lemma.
\begin{lemma}
	Suppose $\Phi\colon M\to\R$ is an equivariant Morse function. There exists
	local coordinates $(x,y_1,\dots,y_{n-1})$ near $p$ and a choice of signs
        $\epsilon_x,\epsilon_{y_1}\dots,\epsilon_{y_{n-1}}\in\{-1,1\}$ such that
		\[\Phi(x,y_1,\dots,y_{n-1})=\Phi(p)+\epsilon_x x^2+\sum_{i=1}^{n-1}\epsilon_{y_j}y_j^2.\]
	\begin{itemize}
		\item if $\tau p\neq p$, then $\tau$ takes $x$ to $x$  and $y_j$ to $y_j$,
		\item if $\tau p=p$, then $\tau$ takes $x$ to $-x$ and $y_j$ to $y_j$.
	\end{itemize}
   \end{lemma}
The above lemma  follows from a more general statement, see~\cite[Lemma 4.1]{Was}.

Wassermann proves also the following result; see~\cite[Lemma 4.8]{Was}
\begin{lemma}
	Equivariant Morse functions are open-dense in the space $C^\infty_{\Z_2}(M;\R)$ of smooth equivariant functions.
\end{lemma}

\begin{corollary}\label{lem:cor}
  The space of doublable boundary Morse functions is open-dense in the space $\cP(N)$ of smooth doublable functions.
\end{corollary}

\subsection{Equivariant jet spaces}\label{sub:ejet}
Equivariant jet spaces are considered  in~\cite{Wall_jet}. There,  $G$  is a compact Lie group.
We recall the general setting and then  specialize to $G=\Z_2$.

 Suppose $V,W$ are two finite dimensional vector spaces over $\R$ with an orthogonal action of $G$.
 Consider $\cE(V)$, the space of smooth functions on $V$. Likewise,  let $\cE(V,W)$ be the space
 of smooth functions from $V$ to $W$. Clearly, $\cE(V,W)$ has a structure of a $\cE(V)$-module.
 We let $\cE^G(V)$ and $\cE^G(V,W)$
be the spaces of $G$-equivariant functions.

Let $\mm(V)$ be the maximal ideal at $0\in V$. We write $\mm^G(V)=\cE^G(V)\cap\mm(V)$ for the maximal ideal
at $0\in V$ of  $G$-invariant functions.
\begin{definition}
  The $G$-equivariant $k$-jet of a function $\Phi\in\cE^G(V,W)$ at $0$ is the class $\mj^k\Phi$ in the space
  \[\mI_G^k(V,W):=\cE^G(V,W)/\mm^G(V)^{k+1}\cdot\cE^G(V,W).\]
\end{definition}
This local description of an equivariant jet space admits a global version.
Suppose  $X,Y$ are both $G$-manifolds. There exists an equivariant jet bundle $\mI_G^k(X,Y)$ over $X\times Y$,
whose fiber at $(x,y)$ is the space $\mI^k_G(V,W)$, where $V$ and $W$ are equivariant
neighborhoods of $x$ in $X$ and $y$ in $Y$ respectively.
Each equivariant function $\Psi\colon X\to Y$ admits an equivariant  jet extension $\mj^k\Psi\colon X\to \mI_G^k(X,Y)$.

As in the classical non-equivariant situation, we define the space of multijets.
Recall that the bundle  $\mI^k_G(X,Y)$   over $X\times Y$,
admits two canonical projections:
\begin{itemize}
\item the \emph{source map} $\ma\colon\mI^k(X,Y)\to X$ and
  \item the \emph{target map} $\mb\colon\mI^k(X,Y)\to Y$.

\end{itemize}
For an integer  $s>1$ we  let $X^{(s)}$ be the subspace of the $s$-fold product
$X^s=X\times\dots\times X$ of $s$-tuples of pairwise distinct points.
We have the product map $\ma^s\colon \mI^k_G(X,Y)^s\to X^s$.
We define $\mI^{k,s}_G(X,Y)$ to be the preimage
\[\mI^{k,s}_G(X,Y):=\left({\ma^s}\right)^{-1}X^{(s)}\subset \mI^k_G(X,Y)^s.\]
Any  $G$-equivariant function $\Psi\colon X\to Y$ has a  $k$-multijet $s$-fold extension
$\mj^{k,s}\Psi\colon X^s\to\mI^{k,s}_G(X,Y)$.
We have the following result.
\begin{theorem}\expandafter{\cite[Theorem 2.1]{Wall_jet}} \label{thm:equiv}
  Suppose $W$ is a smooth submanifold of $\mI_G^{k,s}(X,Y)$. The set of $G$-equivariant functions $C^\infty_G(X,Y)$
  whose multijet extension is transverse to $W$ is residual in $C^\infty_G(X,Y)$.
\end{theorem}

\section{Cerf theory for manifolds with boundary}\label{sec:cerf}
\subsection{Decomposing space of equivariant functions}\label{sub:decomp}
Suppose $M$ is a $\Z_2$-manifold  obtained by doubling $N$ a compact manifold with boundary.
\begin{lemma}\label{lem:block}
  Suppose   $\Phi\colon M\to\R$ is a $\Z_2$-equivariant function, $p\in M^{\Z_2}$. Then
  the matrix of second derivatives $D^2\Phi(p)$   has block structure,
  $D^2\Phi(p)=A\oplus B$, where $A$ is the matrix of derivatives with respect to $y_1,\dots,y_{n-1}$
  and $B$ is the $1\times 1$ matrix with entry $\frac{\partial^2\Phi}{\partial x^2}(p)$.
\end{lemma}
\begin{proof}
  The statement is equivalent to $\frac{\partial^2\Phi}{\partial x\partial y_i}(p)=0$ for all $i=1,\dots,n-1$.
  To see this, consider the function $H_i=\frac{\partial^2\Phi}{\partial x\partial y_i}$. As $\Phi$ is $\tau$-invariant and
  $\tau x=-x$, while $\tau y_i=y_i$, we note that $H_i(\tau z)=-H_i(z)$.
  Now  $p$ is fixed by $\tau$, so  $H_i(p)=0$ as desired.
\end{proof}
Let $\cF$ be the space of all smooth equivariant functions from $M$ to $\R$. 
Inside $\cF$ we will specify several subspaces. Our  convention is that the superscript  refers to codimension of
a given condition. The tilde over a subspace indicates that the space is given by a condition, but the relevant space
will be defined as a relatively open set in the tilded space.
\begin{itemize}
  \item The space $\cF^0$ is the subspace of functions $\Phi\in\cF$ satisfying $\det D^2\Phi(p)\neq 0$ for any point $p\in M$ such that $D\Phi(p)=0$;
  \item The space $\wt{\cF}^1$ is the subspace of $\Phi\in\cF$ satisfying 
    $\det D^2\Phi(p)\neq 0$ for all but one $\Z_2$-orbit of critical points
    however, there exists a single orbit $p_0,\tau p_0\in M$ with $D\Phi(p_0)=0$ such that $\dim \ker D^2\Phi(p_0)=1$;
  \item The space $\cF^{\ge 2}_{1}$ is the space of $\Phi\in\cF$ such that 
    $\det D^2\Phi(p)=0$ for more than $1$ orbit of critical points  of $\Phi$;
  \item The space $\cF^{\ge 2}_{2}$ is the space of $\Phi\in\cF$ such that there exists an orbit of critical points $p,\tau p$ with $D\Phi(p)=0$ and $\dim\ker D^2\Phi(p)>1$.
  \end{itemize}
Note that the spaces $\cF^{\ge 2}_1$ and $\cF^{\ge 2}_{2}$ are not mutually disjoint. 

We decompose the space $\wt{\cF}^1$ into two subspaces
\begin{itemize}
\item $\wt{\cF}^1_{1}$ --- when the point $p_0$ is the interior point of $N$
\item $\wt{\cF}^1_2$ --- when $p_0$ is on the boudary.
\end{itemize}
We describe $\wt{\cF}^1_1$ in more detail.
Let us fix  a vector $v\in T_{p_0}M$ such that $v\in\ker D^2\Phi(p_0)$. We declare
\begin{itemize}
  \item $\Phi\in \cF^1_1$ if $D^3\Phi(p_0)$ evaluated at $(v,v,v)$ is non-zero;
  \item $\Phi\in \cF^{\ge 2}_3$ if $D^3\Phi(p_0)$ evaluated at $(v,v,v)$ is zero.
\end{itemize}

If $p_0\in\partial N$, we need to specify the position of the vector $v$ generating $\ker D^2\Phi(p_0)$ with respect
to the $\tau$-action.
\begin{lemma}\label{lem:dimker1}
  Suppose $\Phi\colon M\to\R$ is an equivariant function such that  $\dim\ker D^2\Phi(p)=1$ for
  a point  $p\in\partial N=M^{\Z_2}$. Let $v\in T_pM$ be a vector
  spanning $\ker D^2\Phi(p)$. Then either $\tau v=v$ or $\tau v=-v$.
\end{lemma}
\begin{proof}
  By Lemma~\ref{lem:block}, there is a decomposition $T_pM=T_p M^{\Z_2}\oplus V$
  with a one-dimensional $V$, such that $D^2\Phi(p)$  has a compatible block structure.
  Here, $\tau$ fixes $T_p M^{\Z_2}$ and acts by multiplication by $-1$ on $V$.

  Write $A$ and $B$ for the two blocks. The dimension of the kernel of a block sum of matrices is the sum
  of dimensions of the kernels of the summands, so either  $\dim\ker A=1$ and $\dim\ker B=0$,
  or vice versa. The first possibility means that $\ker D^2\Phi(p)\subset T_pM^{\Z_2}$, so with $v$ spanning this kernel,
  $\tau v=v$. The other
  possibility is that $\ker D^2\Phi(p)\subset V$, leading to $\tau v=-v$.
\end{proof}
Lemma~\ref{lem:dimker1} allows us to specify two subspaces of $\wt{\cF}^1_2$. We keep using the notation $v$ for a vector generating $\ker D^2\Phi(p_0)$.
\begin{itemize}
  \item The space $\wt{\cF}^1_{2,1}$ is the space of functions in $\wt{\cF}^1_2$ such $p_0\in M^{\Z_2}$ and $\tau v=v$;
  \item The space $\wt{\cF}^1_{2,2}$ is the space of functions in $\wt{\cF}^1_2$ such that $p_0\in M^{\Z_2}$ and $\tau v=-v$.
\end{itemize}
The spaces $\wt{\cF}^1_{2,1}$ and $\wt{\cF}^1_{2,2}$ describe different kinds of degeneracies of functions. 
We set $\wt{\cF}^1_{2,1}=\cF^1_{2,1}\cup \cF^2_3$, where
\begin{itemize}
  \item The space $\cF^1_{2,1}$ is defined by the condition  $D^3\Phi(p_0)(v,v,v) \neq 0$.
  \item The space $\cF^{\ge 2}_3$ is defined by the condition that $D^3\Phi(p_0)(v,v,v)=0$.
\end{itemize}
In local coordinates $(y_1,\dots,y_{n-1},x)$, if $D^2\Phi(p)$ is diagonal and $\frac{\partial^2 \Phi}{\partial y_{n-1}^2}(p_0)=0$,
the two spaces are distinguished by vanishing or not of $\frac{\partial^3 \Phi}{\partial y_{n-1}^3}(p_0)$.

While it is tempting to copy these conditions to the case $\wt{\cF}^1_{2,2}$, we hit the following problem:
by equivariance, $\frac{\partial^3\Phi}{\partial x^3}(p_0)=0$.
Merely studying $\frac{\partial^4 \Phi}{\partial x^4}(p_0)$ is not enough: in the presence of
terms like $\frac{\partial^3\Phi}{\partial x^2\partial y_i}$, the fourth derivative might
depend on the choice of coordinates. Therefore, we make the following definition.
\begin{definition}\label{def:deg_quad}
  Suppose $\Phi\in \wt{\cF}^1_{2,2}$ and $p_0$ is the critical point with $\det D^2\Phi(p_0)=0$.
  The \emph{quasi-homogeneous quadratic form}  $B\Phi(p_0)$ of $\Phi$
  is given by 
	\[B\Phi(a,b_1,\dots,b_{n-1})=\alpha a^2 + \sum_{i=1}^{n-1}\alpha_i ab_i+\sum_{i=1}^{n-1}\alpha_{ij} b_ib_j,\]
  where $\alpha=\frac{\partial^4\Phi}{\partial x^4}(p_0)$, $\alpha_i=\frac{\partial^3\Phi}{\partial x^2\partial y_i}(p_0)$,
  $\alpha_{ij}=\frac{\partial^2\Phi}{\partial y_i\partial y_j}(p_0)$.
\end{definition}
\begin{lemma}\label{lem:degen}
  If $\Phi\in\wt{\cF}^1_{2,2}$, then vanishing of  $\det(B\Phi)$  is independent of coordinate changes.
\end{lemma}
\begin{proof}
	The statement and the proof are similar to the well-known fact that the  quadratic form
	$D^2\Phi(p)$ is well-behaved under the change of local coordinates as long as $D\Phi(p)=0$. 
	For the  proof, let us assign weight $2$ to $y_i$ and weight $1$  to $x$. If $\Phi\in\wt{\cF}^1_{2,2}$,
	we see that $B\Phi(x,y_1,\dots,y_{n-1})$ is the least degree term of the Taylor
	expansion of $\Phi$. 
	Define the matrix $\Gamma=\{\gamma_{ij}\}_{i,j=1}^n$ with
	$\gamma_{ij}=\alpha_{ij}$ if $i,j\le n-1$, $\gamma_{in}=\gamma_{ni}=\alpha_i$,
	$\gamma_{nn}=\alpha$. With $\vec{y}=(y_1,\dots,y_{n-1},x^2)$, we have
	\[B\Phi(x,y_1,\dots,y_{n-1})=\vec{y}\Gamma\vec{y}^T.\]

	Suppose $(\wt{x},\wt{y}_1,\dots,\wt{y}_{n-1})$ are other coordinates
	transformed equivariantly by $y_i=\sum a_{ij}\wt{y}_j+a_j\wt{x}^2+\dots$,
	$x=a\wt{x}+\dots$ with higher order terms meaning higher order terms
	with respect to the weighted degree. Define $C$ as the $n\times n$ matrix of the derivatives
	of the change of variables, that is $C=\{c_{ij}\}$, where $c_{ij}=a_{ij}$
        if $i,j\le n-1$, $c_{nj}=0$, $c_{jn}=a_j$ for $j\le n-1$, $c_{nn}=a$.
	Also, define $E$ to be the matrix as $C$, but with $a^2$ on the $(n,n)$-place.
        We write $C'$ for the $(n-1)\times (n-1)$ minor of $C$ obtained by deleting the last row and the last column.
        It is straightforward to see that $\det C=a\det C'$, $\det E=c^2\det C'$
	so $C$ is non-degenerate if and only if $E$ is.

	It can be verified that
	\[B\Phi(\wt{x},\wt{y}_1,\dots,\wt{y}_{n-1})=\vec{\wt{y}}E\Gamma E^T\vec{\wt{y}}^T,\]
	where $\vec{\wt{y}}=(\wt{y}_1,\dots,\wt{y}_{n-1},\wt{x}^2)$.
\end{proof}
Given Lemma~\ref{lem:degen}, we specify two subspaces of $\wt{\cF}^1_{2,2}$:
\begin{itemize}
  \item The space $\cF^1_{2,2}$ is defined by $\Phi\in\wt{\cF}^1_{2,2}$ for which $B\Phi$ is non-degenerate;
  \item The space $\cF^{\ge 2}_5$ is defined by $\Phi\in\wt{\cF}^1_{2,2}$ for which $B\Phi$ has non-trivial kernel.
\end{itemize}
\subsection{Counting dimensions}\label{sub:counting}
We are now in position to apply Equivariant Transversality Theorem~\ref{thm:equiv} to restrict potential singularities
of a function. To state the result, we introduce one more quantity.

\begin{lemma}\label{lem:parameter_count}
  Suppose $W$ is a submanifold in the jet space $\mI_G^{k,s}(M;\R)$.
  Suppose $\cF_W$ is the subset of equivariant functions whose degree $k$ multijet extension misses $W$. 
  \begin{itemize}
    \item[(a)] If $\codim W>s\dim M$, then $\cF_W$ is residual in $\cF$;
    \item[(b)] If $\codim W=1+s\dim M$, and $\Phi_\sigma$, $\sigma\in[0,1]$, is such that $\Phi_0,\Phi_1$ miss $\cF_W$,
      then $\Phi_\sigma$ can be perturbed rel $0,1$ to a path that hits $\cF_W$ transversally finitely many times;
    \item[(c)] If $\codim W>1+s\dim M$, and $\Phi_\sigma$, $\sigma\in[0,1]$, is such that $\Phi_0,\Phi_1$ miss $\cF_W$,
      then $\Phi_\sigma$ can be perturbed rel $0,1$ to a path that misses $\cF_W$.
  \end{itemize}
\end{lemma}
\begin{proof}
  To prove (a), it is enough to show that if $\Phi\in\cF$ is such that $\mj^{k,s}\Phi$ is transverse to $W$
  at a point $\mj^{k,s}\Phi(z_1,\dots,z_s)$,
  then $\mj^{k,s}\Phi(z_1,\dots,z_s)$ misses $W$. The domain of the map $\mj^{k,s}$ is $M_0^s$, which is
  of dimension $s\dim M$.
  Hence, by transversality, the image of $\mj^{k,s}$ misses $W$.

  An analogous parameter-counting argument works for (b) and (c), however 
  a formal proof of (b) and (c) requires a technical step. It is well-known to the experts, so we provide only a sketch.
  Let $\wt{M}=M\times[0,1]$ with the $G$-action induced from
  the action on $M$ and a trivial action on the interval. We define $\wt{M}^s_0$ to be the set of pairs
  of points $(z_1,\sigma_1),\dots,(z_s,\sigma_s)$ such that $z_i\neq z_j$ if $i\neq j$, but $\sigma_1=\dots=\sigma_s$.
  Then, $\dim\wt{M}=1+s\dim M$. 

  The space $\mI_G^{k,s}(\wt{M};\R)$ fibers over $\mI_G^{k,s}(M;\R)$. This can be seen locally: take the Taylor expansion
  of a function on $M\times[0,1]$ and set all terms corresponding to $[0,1]$-variable to $0$.
  Let $\wt{W}$ be the preimage of $W$ under the map.
  It is a codimension $\codim W$ submanifold of $\mI_G^{k,s}(\wt{M};\R)$.
  Define $\wt{W}_0=\wt{W}\cap\ma^{-1}(\wt{M}^s_0)$.
  Now, $\codim \wt{W}_0=s-1+\codim W$, because equality of the $\sigma_i$ coordinates gives
  $s-1$ conditions.

  Given this notation, a family $\Phi_\sigma$ as in items (b) and (c) induces a map $\wt{\Phi}\colon\wt{M}\to\R$
  via $\wt{\Phi}(z,\sigma)=\Phi_\sigma(z)$. This map has its multijet extension
  \[\mj^{k,s}_G \wt{\Phi}\colon \wt{M}^s\to \mI_G^{k,s}(\wt{M};\R).\]
  Its image is of dimension $\dim\wt{M}^s=s+s\dim M$. If $s-1+\codim W=s+s\dim M$,
  then $\mj^{k,s}_G \wt{\Phi}$ intersects
  $\wt{W}_0$ at finitely many times, each such intersection corresponds to
  a parameter $\sigma$ for which $\mj^{k,s}\Phi_\sigma$ hits $W$.

  If $\dim\wt{M}^s<\codim\wt{W}_0$, then a residual set of maps $\wt{\Phi}\colon\wt{M}\to\R$ has
  multijet extension missing $\wt{W}_0$. This
  implies that $\Phi_\sigma$ can be perturbed rel boundary to a path that misses $W$.
\end{proof}

As a corollary, we have the following result.
\begin{proposition}\label{prop:generic}\
  \begin{itemize}
    \item[(a)] The set of equivariant maps in $\cF^0$ is residual;
    \item[(b)] Any path $\Phi_\sigma$, $\sigma\in[0,1]$ such that $\Phi_0,\Phi_1\in\cF^0$ can be perturbed to a path that
      \begin{itemize}
	\item $\Phi_\sigma\in\cF^0$ for all but finitely many $\sigma$,
	\item  $\Phi_\sigma\in\cF^1_1\cup\cF^1_{2,1}\cup\cF^1_{2,2}$ for finitely many $\sigma$,
	\item $\Phi_\sigma$ misses $\cF^{\ge 2}_1\cup\cF^{\ge 2}_2\cup\cF^{\ge 2}_3\cup\cF^{\ge 2}_4\cup\cF^{\ge 2}_5$.
      \end{itemize}
  \end{itemize}
\end{proposition}
\begin{proof}
  The proof relies on a parameter counting argument and Lemma~\ref{lem:parameter_count}.
  
  First note that the complement of $\cF^0$ is the sum of spaces
  $\wt{\cF}^1_1\cup\wt{\cF}^1_2\cup\cF^{\ge 2}_1\cup\cF^{\ge 2}_2$.
  A case-by-case analysis reveals that $\wt{\cF}^1\cup\wt{\cF}^1_2$ are defined by a codimension $1$ conditions, while
  $\cF^{\ge 2}_1,\cF^{\ge 2}_2$ are defined by codimension $2$ conditions. 

  For $\wt{\cF}^1_1$ this is done as follows. Choose a point $p_0\notin M^{\Z_2}$ and local coordinates
  $y_1,\dots,y_n$   near it. The conditions defining $\wt{\cF}^1_1$ at that point (that is the intersection
  of the manifold $W\subset I^{2,1}_G(M;\R)$ with $\mb^{-1}(p_0)$) are
  \begin{itemize}
    \item $\frac{\partial\Phi}{\partial y_i}(p_0)=0$ for $i=1,\dots,n$;
    \item $\det D^2\Phi(p_0)=0$.
  \end{itemize}
  Therefore, the manifold $W\subset I^{2,1}_G(M;\R)$ is of codimension $n+1$.

  For $\wt{\cF}^1_2$, we choose $p_0\in M^{\Z_2}$ and local coordinates $(x,y_1,\dots,y_{n-1})$
  where $\tau x=-x$ and $\tau y_i=y_i$.
  Then the  conditions at $p_0$ defining $\wt{\cF}^1_2$ are
  \begin{itemize}
    \item $\frac{\partial\Phi}{\partial y_i}(p_0)=0$ for $i=1,\dots,n-1$,
    \item $\det D^2\Phi(p_0)=0$,
    \item $p_0\in \partial M$.
  \end{itemize}
  This means that the subspace $W\subset \mI^{2,1}_G(M;\R)$ defining 
  $\wt{\cF}^1_2$ fibers via $\mb$ over $M^{\Z_2}$ with fiber of codimension $n$. Hence, it is of codimension $n+1$.

  Note that in the latter case $\frac{\partial\Phi}{\partial x}(p_0)$ is satisfied automatically and
  does not increase the codimenision of the  subspace.

  The parameter counting for other subspaces is analogous.
\end{proof}

\subsection{Local forms for the degenerations}\label{sub:local}
We provide a normal form of a function in $\cF^1_1,\cF^1_{2,1}$, and $\cF^1_{2,2}$.
The proofs are classical, however we present sketches of proofs to assert that
the changes of variables can be made equivariant.

We begin with a general lemma.
It can be found in \cite[proof of Lemma 2.1]{Milnor_morse}, as
a step in proving
Morse Lemma. 
\begin{lemma}\label{lem:type_gen}
  Suppose $H$ is a smooth function in variables $u_1,\dots,u_n$. Suppose  $DH(p)=0$ and $D^2H(p)$ is diagonal
  and the term $c=\frac{\partial^2 H}{\partial u_1^2}(p)\neq 0$. Then, there exists a change of variables
  $(u_1,\dots,u_n)\to (w_1,u_2,\dots,u_n)$, local near $p$, such that
  \[H(u_1,\dots,u_n)=c w_1^2+\sum_{2\le j\le k\le n}u_ju_kH'_{jk}(w,u_2,\dots,u_n).\]
  Moreover, if $\tau$ acts on $\R^n$ via $\tau u_i=\epsilon_iu_i$, and $H$ is $\tau$-invariant,
  then the change of variables can be made $\tau$-equivariant.
\end{lemma}
\begin{proof}
  We will say that a function $L\colon M\to\R$ 
  is $\epsilon$-equviariant if $L\circ\tau=\epsilon L$. With this terminology, we assume that $H(p)=0$,
  $p$ has coordinates $(u_1,\dots,u_n)$, and we write
  $H(z)=\int_0^1\frac{dH}{dt}(tz)dt$, leading to:
  \[H(u_1,\dots,u_n)=\sum_{j=1}^n u_j H_j(u_1,\dots,u_n).\]
  Here the functions $H_j$ are $\epsilon_j$-equivariant by construction. 
  Moreover $H_j(0,0,\dots,0)=\frac{\partial H}{\partial u_j}(p)=0$. That is, we may apply Hadamard Lemma again to obtain
  \[H(u_1,\dots,u_n)=\sum_{j\le k} u_ju_k H_{jk}(u_1,\dots,u_n),\]
  where now $H_{jk}$ are $\epsilon_j\epsilon_k$-equivariant.

  We use the fact that $H_{jk}(p)=\frac{\partial^2H}{\partial u_j\partial u_k}$, vanishing if $j\neq k$. Set now
  \[w=\sqrt{\frac1c H_{11}}\left(u_1+\sum_{j>1}u_j\frac{H_{1j}}{H_{11}}\right).\]
  Then, $\frac{\partial w}{\partial u_1}(p)=1$. Moreover $w$ is $\epsilon_1$-equivariant. Hence, map $(u_1,\dots,u_n)\mapsto (w,u_2,\dots,u_n)$
  is an equivariant change of variables as in the  claim.
\end{proof}
From the above result, we obtain the following corollary.
\begin{lemma}\label{lem:type_gen2}
  Suppose $H$ is a smooth function in variables $u_1,\dots,u_n$. Suppose  $DH(p)=0$ and $D^2H(p)$ is diagonal
  with diagonal terms $c_1,\dots,c_n$, and all but the last term are non-zero. Then, there is a $\tau$-equivariant
  base change $(u_1,\dots,u_n)\mapsto (w_1,\dots,w_n)$ so that
  \[H(w_1,\dots,w_n)=c_1w_1^2+\dots+c_{n-1}w^2_{n-1}+w_n^3 H'(w_1,\dots,w_n).\]
\end{lemma}
\begin{proof}
  Apply Lemma~\ref{lem:type_gen} so as to obtain coordinates $u_1',\dots,u_n'$ and a function $H'$ such that
  \[H(u_1',\dots,u_n')=c_1{u_1'}^2+\dots+c_{n-1}{u_{n-1}'}^2+{u_n'}^2H'(u_1',\dots,u_n').\]
  We have that $H_1(0,\dots,0)=0$, so by Hadamard lemma, we obtain
  \[H_1(u_1',\dots,u_n')=u_1'H'_1(u_1',\dots,u_n')+u_2'H_2(u_2',\dots,u_n')+\dots+u_n'H'_n(u_1',\dots,u_n').\]
  We now apply the coordinate change
  \[w_i=u'_i+\frac{1}{2c_i} {u'_n}^2H'_i.\]
  This is a local diffeomorphism. Moreover, $\tau$ acts on ${u'_n}^2H'_i$ in the same way as on $w_i$, so the change is $\tau$-equivariant.
  Its effect is to trade the term ${u_n'}^2u'_iH'_i$ in $H$ for a term starting with ${u'_n}^4$.
  With this change of variables, we eventually obtain:
  \[H(w_1,\dots,w_n)=c_1w_1^2+\dots+c_{n-1}w_{n-1}^2+w_n^3H''(w_1,\dots,w_n).\]
\end{proof}

With the above lemmas, we describe the normal form of a function in $\cF^1_1$.
\begin{proposition}\label{prop:normalF11}
  Suppose $G$ is an equivariant function on $M$,
  the point  $p_0\notin M^{\Z_2}$  is a critical point of $G$, such that $DG(p_0)=0$ and $\dim\ker D^2G(p_0)=1$.
 Assume also, that  there exists a vector $v\in\ker D^2G(p_0)$ satisfying  $D^3G(p_0)(v,v,v)\neq 0$. 
	Then, there exists a local coordinate system $y_1,\dots,y_n$ near $p_0$
	such that
	\[G(y_1,\dots,y_n)=\sum_{i=1}^{n-1}\epsilon_i y_i^2+y_n^3+G(p_0),\]
	where $\epsilon_1,\dots,\epsilon_{n-1}\in\{-1,1\}$.
\end{proposition}
\begin{proof}
	The proof is classical and follows the lines of the proof of classical Morse
	Lemma as in \cite[Lemma 2.1]{Milnor_morse}. First, choose coordinates $y_1,\dots,y_n$
	such that $D^2G(p)$ is diagonal, $\frac{\partial^2G}{\partial y_n^2}(p)=0$. It
	follows from the assumptions that $\frac{\partial^3 G}{\partial y_n^3}(p)\neq 0$. We are in
	the situation of Lemma~\ref{lem:type_gen2}; as $p\notin M^{\Z_2}$, we consider trivial $\tau$-action near $p$. By
	Lemma~\ref{lem:type_gen2}, there exists a change of variables such that
	\[G(y_1,\dots,y_n)=c_1y_1^2+\dots+c_{n-1}y_{n-1}^2+y_n^3G'(y_1,\dots,y_n)+G(p_0).\]
	Here, $c_i$ are the diagonal terms of the matrix  $D^2G(p)$. We can
	make them equal to $\pm 1$, by replacing $y_i$ with $\sqrt{|c_i|}y_i$.

	Note that $G'(0,\dots,0)\neq 0$, because it is equal, up to a constant factor,
	to the third derivative $\frac{\partial^3G}{\partial y_n^3}(p)$. Hence, a change of variables $y_n\mapsto y_n\sqrt[3]{G'(y_1,\dots,y_n)}$ is a local diffeomorphism. After this change, $G$ has the desired form.
\end{proof}
Next, we describe the local form for $\cF^1_{2,1}$. The argument is very similar
to the proof of Proposition~\ref{prop:normalF11}.
\begin{proposition}\label{prop:normalF121}
	Suppose $G\in\cF^1_{2,1}$, that is, there exists point $p_0\in M^{\Z_2}$
	such that $DG(p_0)=0$, $\dim\ker D^2G(p_0)=1$, the vector $v$ spanning
	the kernel of $D^2G(p_0)$ is tangent to $M^{\Z_2}$, and $D^3G(v,v,v)\neq 0$ at
	$p_0$.

	Then, there exist local coordinates $(x,y_1,y_2,\dots,y_{n-1})$ near $p_0$,
	with $\tau x=-x$ and  $\tau y_i = y_i$ such that
	\[G(x,y_1,\dots,y_n)=\epsilon_x x^2+\sum_{i=1}^{n-1}\epsilon_i y_i^2+\epsilon_ny_n^3+G(p_0),\]
  where $\epsilon_x,\dots,\epsilon_n\in\{-1,1\}$.
\end{proposition}
\begin{proof}
	The proof is analogous to the proof of Proposition~\ref{prop:normalF11},
	except that we keep track of the $\tau$-action.
	Apply Lemma~\ref{lem:type_gen2} to $(x,y_1,\dots,y_{n-1})$  to obtain 
	\[G(x,y_1,\dots,y_n)=c_x x^2+c_1y_1^2+\dots+c_{n-2}y_{n-2}^2+y_{n-1}^3H(x,y_1,\dots,y_{n-1}).\]
	We can change $c_i$ to $\sgn c_i$ by rescaling the $y$-variable linearly.
	Now $H(0,\dots,0)=\frac{\partial^3G}{\partial y_{n-1}^3}(p)\neq 0$,
	so the function $h(x,y_1,\dots,y_{n-1}):=H(x,y_1,\dots,y_{n-1})^{1/3}$
	is smooth near $p$. Moreover, it is $\tau$-invariant. The change of coordinates
	$y_{n-1}\mapsto y_{n-1}h(x,y_1,\dots,y_{n-1})$ transforms $G$ to the desired
  form.
\end{proof}
Finally, we investigate the case of $\cF^1_{2,2}$.
\begin{proposition}\label{prop:normalF122}
	Suppose $G\in\cF^1_{2,2}$. That
	is, there exists a point $p_0 \in M^{\Z_2}$
	such that $DG(p_0)=0$, $\dim\ker D^2G(p_0)=1$, the vector $v$ spanning
	the kernel of $D^2G(p_0)$ satisfies $\tau v=-v$, and
	the quasi-homogeneous quadratic form $BG(p_0)$ (see Definition~\ref{def:deg_quad}) is non-degenerate.

	Then, there exist local coordinates $x,y_1,\dots,y_{n-1}$ near
	$p_0$ with $\tau y_i=y_i$, $\tau x=-x$ and
	\[G(x,y_1,\dots,y_n)=\sum_{i=1}^{n}\epsilon_i y_i^2+\epsilon_x x^4,\]
  where $\epsilon_x,\dots,\epsilon_n\in\{-1,1\}$.
\end{proposition}
\begin{proof}
	The proof begins with the same pattern as in Proposition~\ref{prop:normalF122}.
	Apply the linear transformation in the $y$-variables so that the matrix of second derivatives is diagonal. Then,
  apply Lemma~\ref{lem:type_gen2} so that
  \[
  G(x,y_1,\dots,y_n)=c_1y_1^2+\dots+c_ny_n^2+x^3H(x,y_1,\dots,y_n).
  \]
  Notice that the change does not affect the congruency class of the quasi-homogeneous quadratic form $BG$ by Lemma~\ref{lem:degen}.
  Then, $H$ is an antisymmetric function, hence $H(0,y_1,\dots,y_n)=0$.
  That is $H$ is divisible by $x$, so that we can write
  \[
  G(x,y_1,\dots,y_n)=c_1y_1^2+\dots+c_ny_n^2+x^4H'(x,y_1,\dots,y_n).
  \]
  Then, $H'(0,\dots,0)\neq 0$, otherwise the quasi-homogeneous quadratic form $BG$ is degenerate.
  Hence, with $\epsilon_x$ being the sign of $H'(0,\dots,0)$,
  we can apply a coordinate change
  \[x\mapsto x\sqrt[4]{\epsilon_x H'(x,y_1,\dots,y_n)}\]
  This is a $\tau$-equivariant coordinate change.
\end{proof}
\subsection{Unfoldings of $\cF^1_1$, $\cF^1_{2,1}$ and $\cF^1_{2,2}$}\label{sub:beefurcate}
Having defined a normal form of each of the three singularities, we pass to describing
local unfoldings. Later, in Subsection~\ref{sub:versal}, we will show that this unfoldings give a typical behavior
of a 1-parameter family of functions
near each of the three spaces.
\begin{definition}\label{def:stF11}
	A \emph{standard unfolding} of a $\cF^1_1$ singularity
	is the family
	\[G_\lambda(y_1,\dots,y_n)=\sum_{i=1}^{n-1}\epsilon_i y_i^2+y_n^3+\lambda y_n,\]
	where $\lambda\in[-1,1]$; see Figure~\ref{fig:stF11}.

	We call this unfolding \emph{birth/death of a pair of interior critical points}.
\end{definition}
We note that $G_\lambda$ has no critical points if $\lambda>0$, $G_\lambda$ hits
$\cF^1_1$ at $\lambda=0$ and there are two critical points of $G_\lambda$ for $\lambda<0$. These critical points have indices differing by $1$. We say that the standard unfolding of a $\cF^1_1$ critical point is a birth-death of an interior critical point.
\begin{figure}
  \begin{tikzpicture}
    \draw (-5,0) node {\includegraphics[width=4cm]{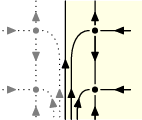}};
    \draw (5,0) node {\includegraphics[width=3cm]{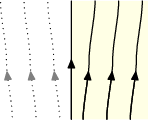}};
    \draw (0,0) node {\includegraphics[width=4cm]{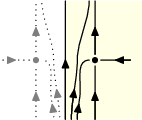}};
  \end{tikzpicture}
  \caption{Standard unfolding of a $\cF^1_1$ singularity. Left: $\lambda<0$. Middle: $\lambda=0$. Right: $\lambda>0$. We have
  drawn exemplary gradient flow of $G_\lambda$. The dotted part on each picture is the symmetric copy, indicating the $\tau$-action.}\label{fig:stF11}
\end{figure}

We pass to $\cF^1_{2,1}$.
\begin{definition}\label{def:stF121}
	A \emph{standard unfolding} of a $\cF^1_{2,1}$ singularity
	is the family
	\[G_\lambda(x,y_1,\dots,y_{n-1})=\epsilon_xx^2+\sum_{i=1}^{n-2}\epsilon_i y_i^2+y_{n-1}^3+\lambda y_n,\]
	where $\lambda\in[-1,1]$; see Figure~\ref{fig:stF121}.

	We call this unfolding \emph{birth/death of a pair of boundary critical points}.
\end{definition}
For $\lambda>0$ there are no critical points. For $\lambda=0$, $G_\lambda\in\cF^1_{2,1}$. For $\lambda<0$, there is a pair of critical points on $M^{\Z_2}$: they
are either both boundary stable or boundary unstable, depending on $\epsilon_x$. Their indices differ by $1$. 
We say that the family $G_\lambda$ is a birth-death of a boundary critical point.
\begin{figure}
  \begin{tikzpicture}
    \draw (-5,0) node {\includegraphics[width=4cm]{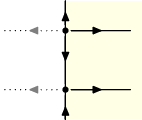}};
    \draw (5,0) node {\includegraphics[width=3cm]{pics-6.eps}};
    \draw (0,0) node {\includegraphics[width=4cm]{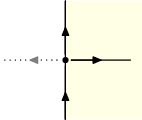}};
  \end{tikzpicture}
  \caption{Standard unfolding of a $\cF^1_{2,1}$ singularity. Left: $\lambda<0$. Middle: $\lambda=0$. Right: $\lambda>0$. We present the flow
  of $\nabla G_\lambda$.}\label{fig:stF121}
\end{figure}

Finally, we define deformation of the $\cF^1_{2,2}$ singularity.
\begin{definition}
	A \emph{standard unfolding} of a $\cF^1_{2,2}$ singularity is the family
	\[G_\lambda(x,y_1,\dots,y_{n-1})=\sum_{i=1}^{n-1}\epsilon_i y_i^2+\epsilon_x(x^4+\lambda x^2),\]
	where $\lambda\in[-1,1]$; see Figure~\ref{fig:stF122}.

	We call this unfolding \emph{collision of a boundary and an interior critical points}.
\end{definition}
\begin{figure}
  \begin{tikzpicture}
    \draw (3,0) node {\includegraphics[width=4cm]{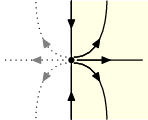}};
    \draw (-3,0) node {\includegraphics[width=4cm]{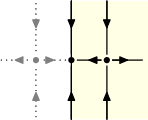}};
  \end{tikzpicture}
  \caption{Standard unfolding of a $\cF^1_{2,2}$ singularity. Left: $\lambda<0$. Right: $\lambda>0$. We present the flow
  of $\nabla G_\lambda$. Case $\lambda=0$ differs from $\lambda>0$ only by the speed of convergence to the critical point in
the $x$-direction.}\label{fig:stF122}
\end{figure}
This deformation is not as standard as the previous deformations, so we explain its criticial points in greater detail.
We assume that $\epsilon_x=1$, changing $\epsilon_x$ to $-1$ results in making
boundary stable critical points boundary unstable and vice versa.

Clearly, $G_\lambda$ is $\tau$-invariant. The critical points of $G_\lambda$ all
have $y_1=\dots=y_{n-1}=0$. There is always a critical point $(0,\dots,0)$, which we call $p_0$. 
If $\lambda>0$, $p_0$ is boundary unstable. If $\lambda<0$, $p_0$ is boundary stable, because the stability depends on the sign of $\epsilon_x\lambda$. For $\lambda\ge 0$ there are no other
critical points. For $\lambda<0$, there are two critical points away from $M^{\Z_2}$
with $x=\pm\sqrt{-\lambda}$. Call them $p_-$ and $p_+$. While regarding $M$ as the double of $N$, precisely one of the two points
is seen in $N$.

The second derivative of $G_\lambda$ restricted
to the space spanned by $\frac{\partial}{\partial y_1},\dots,\frac{\partial}{\partial y_{n-1}}$ is independent of $\lambda$, and the same
for all critical points $p_-,p_+,p_0$.
The behavior of $D^2G_\lambda$ in the $\frac{\partial}{\partial x}$ direction is not fixed and depends on both $\lambda$ and the critical point.

Assuming $\epsilon_x>0$, the index of $p_-$, $p_+$ points is equal to the index of $p_0$ for $\lambda>0$
and is one less than the index of $p_0$ for $\lambda<0$ (if $\epsilon_x<0$, then the index of $p_0$ for $\lambda<0$ is one greater
than the index of $p_{\pm}$).

\subsection{Versal deformations}\label{sub:versal}
In Subsection~\ref{sub:beefurcate}, we described local deformations
of singularities $\cF^1_1,\cF^1_{2,1},\cF^1_{2,2}$. We now prove the following statement.
\begin{proposition}\label{prop:versal}
	Suppose $\Phi_\sigma$ is a family of equivariant Morse functions on $M$.
	Assume that $\Phi_\sigma$ hits $\cF^1_1,\cF^1_{2,1},\cF^1_{2,2}$
	at a parameter $\sigma_0$. Then, there exists a local coordinate system
	near $p_0$, possibly depending on $\sigma$, the non-Morse critical point of $f_{\sigma_0}$, and a function
	$\lambda(\sigma)$ with $\lambda(\sigma_0)=0$ such that
	in these local coordinates $\Phi_\sigma=\Psi_\sigma\circ G_{\lambda(\sigma)}$,
	where $\Psi_\sigma$ is a diffeomorphism of $\R$ for all $\sigma$ near $\sigma_0$. Moreover,
	$\Phi_\sigma$ is transverse (as a path) to $\cF^1_1,\cF^1_{2,1},\cF^1_{2,2}$,
	if $\frac{d\lambda(\sigma)}{d\sigma}(\sigma_0)\neq 0$.
\end{proposition}
\begin{proof}
  We use standard arguments, used e.g. in \cite{Wall_proj} and based on versality theory of \cite[Section 8]{AVG}.
	The statement is by definition equivalent to saying that $G_\lambda$
	constructed in Subsection~\ref{sub:beefurcate} is versal. Indeed, versality means precisely that every deformation
	can be induced from the given one.

	In order to show that the deformation is versal, it is enough
	to show that it is infinitezimally versal by Versality Theorem \cite[Theorem 8.3]{AVG}.\footnote{The proof
	in \cite{AVG} is given in a non-equivariant setting, but it is easily adapted
	to the case with a $\Z_2$-action.}

      Versality is usually discussed locally, i.e. via germs of maps, rather than for global spaces.
      Consider the space $\cO$ of germs at $p_0$ of smooth equivariant functions from $M$ to $\R$ taking
      value $0$ at $0$. On this space  two  groups of  diffeomorphisms act.
      First one is the group of reparametrizations (germs of self-diffeomorphisms of $M$ preserving $p_0$),
      the other is the group of self-diffeomorphism of $\R$ preserving $0$.
      For any $F_0\in\cO$,  let $\cG_{F_0}$ be the orbit	
      of $F_0$ under the action of the product of the two groups (in \cite{AVG} this is called the orbit of RL-equivalence).

      A deformation $F_\lambda$, $\lambda\in\Lambda$, where $\Lambda$ is an open subset of some $\R^k$ ($\Lambda$ is the parameter space)
      defines a (germ of a map) $\cT\colon \Lambda\to\cO$. We say that $F_\lambda$ is infinitezimally versal,	if $\cT$ is transverse to $\cG_{F_0}$. 

      Infinitezimal versality of each of the three deformations of Subsection~\ref{sub:beefurcate} can be readily verified using
      \cite[Section 8.2, Theorem]{AVG}. We present a more conceptual approach.

      Suppose $G_\lambda$, $\lambda\in(-\varepsilon,\varepsilon)$ is as in Definition~\ref{def:stF11}. 
      It deforms the $\cF^1_1$ singularity $G_0$. Assume $W\subset \mI^{3,1}_G(M;\R)$ is the defining set of that singularity.
      By direct inspection of Definition~\ref{def:stF11}, the map $(-\varepsilon,\varepsilon)\times M \to \mI^{3,1}_G(M;\R)$
      given by $(\lambda,z)\mapsto \mj^{3,1}_G G_\lambda(z)$ is transverse to $W$ at $\lambda=0$.
      This, however, is not yet the same as transversality of $\cT$ to the orbit $\cG_{G_0}$ required by the condition for infinitezimal versality.

        Consider the space $\cW\subset\cO$
	of all germs $H\in\cO$ such that $\mj^{3,1}_G H\in W$. Transversality of $\mj^{3,1}_G G_\lambda$
        to $W$ is equivalent to transversality of $\cT$ to $\cW$.

	The key result from Proposition~\ref{prop:normalF11}
        (and Proposition~\ref{prop:normalF121}, Proposition~\ref{prop:normalF122})
	is that each element of $\cW$ has its normal form specified by the number of occurences of $+1$
        in the set $\{\epsilon_1,\dots,\epsilon_{n-1}\}$. That is, $\cW$ decomposes as
        a disjoint sum $\cW_1,\dots,\cW_{n-1}$, and each of the $\cW_i$ is an orbit of a single function. That is,
	transversality to $\cW$ is equivalent to transversality to the orbit $\cG_{G_0}$.
        In other words, $\cT$ is transverse to $\cG_{G_0}$.
	
	An analogous argument work for $\cF^1_{2,1}$ and~$\cF^1_{2,2}$.
\end{proof}
We remark that if the normal form depends on continuous parameters (i.e. there are so-called \emph{moduli}), 
the situation might be more subtle. If $\cW$ is a union of infinitely many orbits, transversality to $\cW$ does not imply
transversality to the orbit $\cG_{G_0}$.

\section{Theorem~\ref{thm:main} and examples}\label{sec:pmain}
\subsection{Proof of Theorem~\ref{thm:main}}\label{sub:proof}
For the reader's convenience, we recall the statement of Theorem~\ref{thm:main}.
\begin{theorem*}
	Suppose $F_0,F_1$ are two Morse functions for a manifold $N$ with boundary.
	Assume $F_\sigma$, $\sigma\in[0,1]$ is a path connecting these functions.
        Then, up to perturbing $F_\sigma$ rel $F_0,F_1$, we can assume that
	\begin{itemize}
        \item there exists a finitely many points $\sigma_1,\dots,\sigma_m$ such that $F_\sigma$ is Morse
          if $\sigma$ is different from $\sigma_1,\dots,\sigma_m$;
	\item at each of $\sigma_i$, there is precisely one non-Morse critical point.
          The change between $F_{\sigma_i-\varepsilon}$ and $F_{\sigma_i+\varepsilon}$
          (with $\varepsilon>0$ sufficiently small) is one of the following:
			\begin{itemize}
				\item birth/death of critical points in the interior;
				\item birth/death of critical points on the boundary;
				\item a collision of a boundary and interior critical point.
			\end{itemize}
	\end{itemize}
      \end{theorem*}
\begin{proof}
  First, by Proposition~\ref{prop:dbldbl} we find doublable functions $P_0$ and $P_1$ such that $P_0$
	has the same critical points as $F_0$ and $P_1$ has the same critical points
	as $F_1$. Moreover, there is a path of Morse functions connecting $F_0$ with $P_0$ and $F_1$ with $P_1$ (see \ref{item:Ffam} of
	Proposition~\ref{prop:dbldbl}).
	Therefore, it is enough to prove the statement for $P_0$ and $P_1$.

	Let $\Phi_0$ and $\Phi_1$ be the equivariant functions on $M$ obtained by doubling $N$. Then, $\Phi_0$ and $\Phi_1$ can be connected by a path of equivariant functions. Call this path $\Phi_\sigma$. As $\Phi_0$, $\Phi_1$ are Morse, Proposition~\ref{prop:generic} applies.
	That is, the path $\Phi_\sigma$ can be perturbed to a path crossing transversely $\cF^1_1$, $\cF^1_{2,1}$ and $\cF^1_{2,2}$,
	and missing $\cF^{\ge 2}_{i}$ for all $i=1,\dots,5$.
	We apply Proposition~\ref{prop:versal} to show that the change of $\Phi_\sigma$ at the intersection points
	is as in Subsection~\ref{sub:beefurcate}: either it is a birth/death of interior critical points, or it is a birth/death of critical points on $M^{\Z_2}$, or it is
	a collision of critical points.

	Define $P_\sigma=\Phi_\sigma|_N$.
	 Then, $P_\sigma$ connects $P_0$ with $P_1$ and the events for $P_\sigma$
	are exactly as described in the itemized list.
\end{proof}

\subsection{Moving critical points to the boundary}\label{sub:moving}
As an  application, we describe the splitting of an interior critical point described in \cite{BNR,BMi}

In \cite{BNR}, the process was explained as a deformation of a $D^4_-$ singularity.
To simplify the notation, we work with $N=\R_{\ge 0}\times\R$, $M=\R^2$.
The higher dimensional case involves adding quadratic terms to the function $\Phi$ and to its deformation,
which  does not substantially differ from the two-dimensional case.

On $M$ we consider coordinates $x,y$, with $\tau x=-x$ and $\tau y=y$. Let 
\[\Phi(x,y)=y^3-yx^2.\]
This function has $D^4_-$ singularity, see \cite[Section 3]{AVG}. The deformation
\[\Phi_\lambda(x,y)=\Phi(x,y)+\lambda y\]
results in pushing critical points to the boundary. Indeed,
\[D\Phi_\lambda=(-2xy,\lambda-x^2+3y^2)\]
If $\lambda>0$, then the critical points are at $(\pm\sqrt{\lambda},0)$. If 
$\lambda<0$, the critical points are at $(0,\sqrt{-\lambda/3})$. This means that
 $\Phi_\lambda|_N$ has a pair of critical points on $\partial N=\{x=0\}$ for $\lambda<0$.
On the other hand,  $\Phi_\lambda|_N$ has a single critical point for $\lambda>0$.

We make the following observation.
\begin{lemma}
	The function $\Phi$ belongs to $\cF^{\ge 2}_2$.
\end{lemma}
\begin{proof}
	The space $\cF^{\ge 2}_2$ has been defined by the condition that there exists
	a critical point $p$ for which $\dim\ker D^2\Phi(p)>1$. And this is the case
	for $\Phi$. At $(0,0)$ all the quadratic terms vanish, so $\dim\ker D^2\Phi(0,0)=2$.
\end{proof}
As $\Phi$ belongs to $\cF^{\ge 2}_2$, one cannot expect that the deformation $\Phi_\lambda$
to be versal (the expectation is that the base of a versal deformation has dimension at least $2$).
Consider
\[\Phi_{\lambda,\mu}(x,y)=y^3-x^2y+\lambda y+\mu x^2.\]
We study this deformation and show how we can pass from $\Phi_{-1,0}$ to $\Phi_{1,0}$
by hitting only $\cF^{1}_1,\cF^{1}_{2,1},\cF^{1}_{2,2}$ strata.

To this end, we study critical points of $\Phi_{\lambda,\mu}$. Note that
\begin{equation}\label{eq:Dphi} D\Phi_{\lambda,\mu}(x,y)=(2x(\mu-y),3y^2-x^2+\lambda).\end{equation}
It is also convenient to compute the second derivative.
\begin{equation}\label{eq:D2phi}
  D^2\Phi_{\lambda,\mu} = 
  \begin{pmatrix} 
    2(\mu-y) & -2x \\ -2x & 6y 
  \end{pmatrix}.
\end{equation}
To study critical points of $\Phi_{\lambda,\mu}$, following \eqref{eq:Dphi}, we consider two cases $x=0$ and $y=\mu$. It might be helpful to consult Figure~\ref{fig:crit}.
\begin{figure}
  \includegraphics[width=5cm]{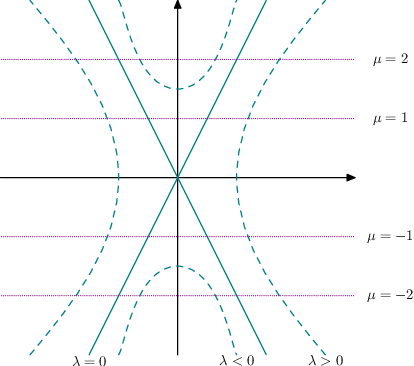}
  \caption{The behavior of $\Phi_{\lambda,\mu}$.
    The derivative $\frac{\partial\Phi_{\lambda,\mu}}{\partial\mu}$ vanishes on $x=0$ and on one of the horizontal lines, whose height is $\mu$.
    The derivative $\frac{\partial\Phi_{\lambda,\mu}}{\partial\lambda}$ vanishes on one of the hyperbolas (or a degenerate hyperbola)
    depending on the parameter $\lambda$. Critical points of $\Phi_{\lambda,\mu}$ are read off as the intersection points of the two vanishing sets. Transverse intersection
  corresponds to Morse-type critical points.}\label{fig:crit}
\end{figure}

\emph{Case 1.} Suppose $x=0$. Assume first $\lambda<0$.
There are two critical points $p_{\pm}=(0,\pm \sqrt{-\lambda/3})$.
As $x=0$, by \eqref{eq:D2phi}, the second derivative is diagonal.
The indices $\pind p_{\pm}$ of the two critical points are $0$ and $1$ depending
on whether $y>0$ or $y<0$. The second derivative with respect to $x$ is $2(\mu-y)$.
That is, for $\mu\neq \pm\sqrt{-\lambda/3}$, both $p_+$ and $p_-$ are non-degenerate, i.e. they are Morse critical points.

If $\lambda=-3\mu^2$, then the kernel either of $D^2\Phi_{\lambda,\mu}(p_+)$ (if $\mu>0$) or of $D^\Phi_{\lambda,\mu}(p_-)$ (if $\mu<0$) 
is one-dimensional, but tangent to $x=0$. By direct inspection $B\Phi$ is non-degenerate.
With these parameters $\Phi_{\lambda,\mu}$ belongs to $\cF^1_{2,2}$.

If $\lambda=0$ but $\mu\neq 0$, then $p_+=p_-=(0,0)$ and so $D^2\Phi$ has one-dimensional kernel, spanned
by $\frac{\partial}{\partial y}$. The derivative $\frac{\partial^3}{\partial y^3}\Phi_{\lambda,\mu}\equiv 6\neq 0$,
so $(0,0)$ corresponds to $\cF^1_{2,1}$ stratum.

For $\lambda>0$, there are no critical points with $x=0$.


\emph{Case 2.} Suppose $y=\mu$. Then, the critical points have
$x^2=\lambda+3\mu^2$. For $\lambda>-3\mu^2$, we have a pair of critical
points $q_\pm=(\pm\sqrt{\lambda+3\mu^2},\mu)$ that are swapped under the $\tau$-action. So on $N$, we have one interior critical point. If $\lambda=-3\mu^2$, $q_+=q_-$, $x=0$ and we are in the situation
of Case~1.

In Figure~\ref{fig:crit}, the critical points for $y=\mu$ correspond to intersections of a green, dashed line with a horizontal purple line.

From this discussion it follows that in the space of parameters we can specify
the following four regions (see Figure~\ref{fig:philambda}):
\begin{itemize}
	\item $\lambda>0$ corresponds to a single critical point in the interior.
	\item $\lambda<0$ but $\lambda>-3\mu^2$ has a pair of critical points
		on the boundary (of the same stability) and an interior critical point. This case accounts for two
		regions with $\mu>0$ and $\mu<0$.
	\item $\lambda<-3\mu^2$, we have two critical points on the boundary that
		have opposite stability.
\end{itemize}
When starting from parameter $\lambda=1,\mu=0$, and going counterclockwise, we hit the vertical line $\lambda=0$.
This creates a cancelling pair of critical points on the boundary, both of which are boundary unstable. Next, we hit the parabola $\lambda=-3\mu^2$.  At that point, a collision occurs. The interior critical point hits the upper boundary critical point changing its stability. We reach
$\lambda=-1,\mu=0$ parameter.

When going clockwise, we hit the line $\lambda=0$ first, but this time with parameter $\mu<0$. A cancelling pair of critical points is
created, but this time they are boundary stable. Then, on hitting the parabola, a collision occurs, changing the stability of the \emph{lower} boundary critical point.
\begin{figure}
  \begin{tikzpicture}
    \node at (0,0) {\includegraphics[width=8cm]{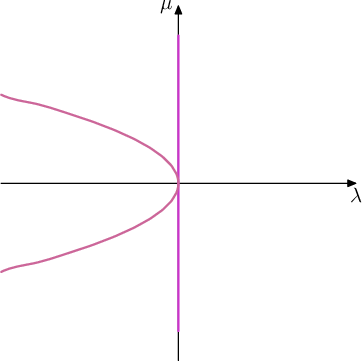}};
    \node[opacity=0.7] at (-4,0) {\includegraphics[width=2cm]{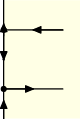}};
    \node[opacity=0.7] at (4,1.5) {\includegraphics[width=2cm]{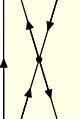}};
    \node[opacity=0.7] at (-1.5,3) {\includegraphics[width=2cm]{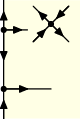}};
    \node[opacity=0.7]  at (-1.5,-3.5) {\includegraphics[width=2cm]{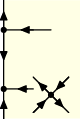}};
  \end{tikzpicture}
  \caption{Bifurcation of $\Phi_{\lambda,\mu}$. The discriminant set $\{\lambda=0\}\cup\{\lambda=-3\mu^2\}$ is drawn. For each of the four regions, we sketch the gradient vector field near the critical locus of $\Phi_{\lambda,\mu}$.}\label{fig:philambda}
\end{figure}

\bibliographystyle{alpha}
\def\MR#1{}
\bibliography{biblio}

\begin{thebibliography}{AGZV12}

\bibitem[AGZV12]{AVG}
V.~I. Arnold, S.~M. Gusein-Zade, and A.~N. Varchenko.
\newblock {\em Singularities of differentiable maps. {V}olume 1}.
\newblock Modern Birkh\"auser Classics. Birkh\"auser/Springer, New York, 2012.
\newblock Classification of critical points, caustics and wave fronts, Translated from the Russian by Ian Porteous based on a previous translation by Mark Reynolds, Reprint of the 1985 edition.

\bibitem[BM25]{BMi}
Maciej Borodzik and Marcin Mielniczuk.
\newblock Merging boundary critical points of a {M}orse function.
\newblock {\em Proc. Amer. Math. Soc. Ser. B}, 12:64--77, 2025.

\bibitem[BNR16]{BNR}
Maciej Borodzik, Andr\'as N\'emethi, and Andrew Ranicki.
\newblock Morse theory for manifolds with boundary.
\newblock {\em Algebr. Geom. Topol.}, 16(2):971--1023, 2016.

\bibitem[BP16]{BP}
Maciej Borodzik and Mark Powell.
\newblock Embedded {M}orse {T}heory and {R}elative {S}plitting of {C}obordisms of {M}anifolds.
\newblock {\em J. Geom. Anal.}, 26(1):57--87, 2016.

\bibitem[Cer70]{Cerf}
J.~Cerf.
\newblock La stratification naturelle des espaces de fonctions diff\'erentiables r\'eelles et le th\'eor\`eme de la pseudo-isotopie.
\newblock {\em Inst. Hautes \'Etudes Sci. Publ. Math.}, (39):5--173, 1970.

\bibitem[Haj81]{Haj}
B.~Hajduk.
\newblock Minimal {$m$}-functions.
\newblock {\em Fund. Math.}, 111(3):179--200, 1981.

\bibitem[JR72]{JR}
A.~Jankowski and R.~Rubinsztein.
\newblock Functions with non-degenerate critical points on manifolds with boundary.
\newblock {\em Comment. Math. Prace Mat.}, 16:99--112, 1972.

\bibitem[KM07]{KM}
Peter Kronheimer and Tomasz Mrowka.
\newblock {\em Monopoles and three-manifolds}, volume~10 of {\em New Mathematical Monographs}.
\newblock Cambridge University Press, Cambridge, 2007.

\bibitem[Mil63]{Milnor_morse}
J.~Milnor.
\newblock {\em Morse theory}, volume No. 51 of {\em Annals of Mathematics Studies}.
\newblock Princeton University Press, Princeton, NJ, 1963.
\newblock Based on lecture notes by M. Spivak and R. Wells.

\bibitem[Wal85]{Wall_jet}
C.~T.~C. Wall.
\newblock Equivariant jets.
\newblock {\em Math. Ann.}, 272(1):41--65, 1985.

\bibitem[Wal08]{Wall_proj}
C.~T.~C. Wall.
\newblock Projection genericity of space curves.
\newblock {\em J. Topol.}, 1(2):362--390, 2008.

\bibitem[Was69]{Was}
Arthur~G. Wasserman.
\newblock Equivariant differential topology.
\newblock {\em Topology}, 8:127--150, 1969.

\end{thebibliography}

\end{document}